\documentclass[11pt,reqno]{amsart}
\usepackage{amssymb,latexsym,amsmath}
\usepackage{graphicx}
\usepackage{hyperref}
\usepackage{url}		
\usepackage{geometry}
\newcommand{\zc}{{\mathcal Z}}
\newcommand{\z}{{\zeta_{\mathcal{Z}}}}

\theoremstyle{plain}
\numberwithin{equation}{section}
\newtheorem{thm}{Theorem}[section]
\newtheorem{theorem}[thm]{Theorem}
\newtheorem{lemma}[thm]{Lemma}
\newtheorem{remark}[thm]{Remark}
\newtheorem{corollary}[thm]{Corollary}
\newtheorem{example}[thm]{Example}
\newtheorem{definition}[thm]{Definition}

\begin{document}


%
%

\title{Structural Properties of Multiple Zeta Values
}

\author{Tanay Wakhare}

\address{Department of Electrical Engineering and Computer Science, MIT,\\
Cambridge, MA 02139, USA\\
}
\email{twakhare@mit.edu}

\author{Christophe Vignat}
\address{LSS CentraleSupelec, Univerit\'{e} Paris Saclay,\\
Gif sur Yvette, 91192, France\\
christophe.vignat@u-psud.fr 
Department of Mathematics, Tulane University,\\
New Orleans, LA 70118, USA\\
}
\email{cvignat@tulane.edu}
\maketitle


\begin{abstract}
We study some classical identities for multiple zeta values and show that they still hold for zeta functions built from an arbitrary sequence of nonzero complex numbers. We introduce the \textit{complementary zeta function} of a system, which naturally occurs when lifting identities for multiple zeta values to identities for quasisymmetric functions.
\end{abstract}



\section{Introduction and Notation}

Multiple zeta values (MZVs) are simply defined, yet ubiquitous in modern number theory and physics. The MZV of depth $k$ is defined as
\begin{equation}\label{1.1}
\zeta({s_1,\ldots,s_k}):= \sum_{n_1>n_2>\cdots>n_k\ge 1} \frac{1}{n_1^{s_1}n_2^{s_2}\cdots n_k^{s_k}} ,
\end{equation}
and in general exists for positive integral $s_i$ with $s_1\geq 2$ and $s_i\geq 1$ for $2\leq i \leq k$. Throughout this paper, we will refer to \eqref{1.1} as the \textit{Riemann multiple zeta function}, as an analogy to the Riemann zeta function. MZVs were first explored by Euler, but have only been systematically studied for the past few decades, for example by Zagier \cite{Zagier} and Hoffman \cite{Hoffman}. There are many nontrivial linear dependence relations between MZVs, and characterizing all such relations is an extremely difficult (and still unsolved) question.


%
Our main innovation is to consider the following object, which reduces to the Riemann multiple zeta function in the case $z_n = n$.
\begin{definition}
Given a (possibly finite) set of complex numbers $\zc := \{z_{n}\neq 0\}$, ordered in increasing order of magnitude, we define the \textbf{extended zeta function}
\[
\z(s):=\sum_{n=1}^{\infty}\frac{1}{z_{n}^{s}},
\]
and \textbf{extended multiple zeta function}
\[
\z(s_{1},\ldots,s_{k}):=\sum_{n_{1}>n_{2}>\cdots>n_{r}\geq1}\frac{1}{z_{n_{1}}^{s_{1}}\cdots z_{n_{k}}^{s_{k}}}.
\]
\end{definition}

The numbers $z_k$ may be chosen, for example, as the zeros of a given function $\mathcal{Z}\left( z \right)$, but this assumption is not necessary in the rest of the paper. Note that the choice of the zeros of $\mathcal{Z}\left( z \right) = \frac{1}{\Gamma(1-z)}$ yields the usual Riemann zeta and multiple zeta values, as this is an entire function with zeros at $\{1, 2, 3, \ldots\}$. Another particularly nice special case is $z_k=2k+1$, where we form a multiple sum over only odd integers. These ``multiple $t$-values" \cite{HoffmanOdd} are higher order analogs of the Dedekind eta function and behave quite differently from Riemann MZVs. Our unified treatment of both is new, and one of the primary reasons for studying these extended MZVs.

These extended MZVs are very relevant in physics. For example, let $\{ z_n \}$ denote the sequence of energy eigenvalues of a system, which may not have compact closed forms. However, the associated zeta function may still have a nice closed form in terms of well-known transcendental functions. This 
is the case of the ``quantum bouncer'' \cite{Crandall}, where the energy states of the system are zeros of Airy functions without explicit expression, whereas the associated zeta function at positive integers has values in terms of known transcendental constants.

There exist many identities satisfied by MZVs, due to their high degree of symmetry. Our prototype is the simple and elegant Euler identity
\begin{equation}
\zeta\left(2,1\right)=\zeta\left(3\right).
\end{equation}

The question we address in this paper is: \textit{what identities are structural, meaning that they still hold when the Riemann MZV is replaced by the extended MZV?}

This discussion includes the case where the sequence $\{z_{n}\}$ is finite, so that we also consider finite-type sums. In Corollary \ref{eulercor} we will see the structural $\z(2,1) = \tilde{\zeta}_\zc(3)$, where the \textbf{complementary zeta function} $\tilde{\zeta}_\zc(s)$ will be defined in Section \ref{compsec}. A prototypical example of a structural identity for multiple zeta values would be the reflection formula introduced by Euler,
\begin{equation}
\label{reflection_zeta}
\zeta\left( s,t \right) + \zeta\left( t,s \right) 
+\zeta\left( s+t \right)
=
\zeta\left( s \right)\zeta\left( t \right)
\end{equation}
which naturally lifts to 
\begin{equation}
\label{reflection}
\z\left( s,t \right) + \z\left( t,s \right) 
+\z\left( s+t \right)
=
\z\left( s \right)\z\left( t \right)
\end{equation}
after rewriting the domain of summation. Another more elaborate structural relation concerns star-MZVs, which are defined by
\[
\zeta^{\star}\left( s_1,\dots,s_k \right) =
\sum_{n_1\ge n_2\ge \dots \ge n_k \ge 1}
\frac{1}{n_{1}^{s_1}n_{2}^{s_2}\dots n_{k}^{s_k}}
\]
and can be expressed as
\[
\zeta^{\star}\left( s_1,\dots,s_k \right) =
\sum \zeta\left( s_1 \square \dots \square s_k \right),
\]
where the sum is over the $2^{k-1}$ configurations obtained by choosing $\square = ``+"$ or $\square= ``,"$. This identity also naturally lifts to the extended $\z$ and $\z^{\star}$ functions, since it only expresses a symmetry of the domain of summation.

We can recast this question in a more algebraic framework. We have skipped many technical details in the following construction, and the reader is referred to \cite[Section 4]{Henderson} for a complete exposition. Begin with a (possible finite) set of indeterminates $\{z_n\}$, and let $\lambda = (\lambda_1, \lambda_2,\cdots)$ be a composition of $n$. Then we can define the monomial symmetric function $m_\lambda:=\sum_{\sigma \in S_\lambda} z_{\sigma(1)}^{\lambda_1} z_{\sigma(2)}^{\lambda_2} \cdots$, where the sum is over the set of permutations giving distinct terms in the sum (so that the coefficient of any monomial in the sum is simply $1$). We then denote by $QSym$ the graded $\mathbb{Q}$-algebra of \textit{quasisymmetric functions}, which is spanned by the monomial symmetric functions. Under the evaluation homomorphism $z_k \mapsto k$ we have $m_\lambda \mapsto \zeta(\lambda_1,\lambda_2,\ldots)$. The extended MZV naturally lives in the algebra of quasisymmetric functions, $QSym$.


Our key observation is that many of the classical theorems about MZVs are only true after applying this evaluation map. We wish to lift these to theorems in $QSym$, so that by applying other evaluation maps such as $z_k \mapsto 2k+1$ we can effortlessly obtain results about other zeta functions.

Our overarching philosophy is that asymmetry in the indices of MZVs makes them hard to study, so we must first symmetrize them somehow. Results for such symmetrized MZVs should then hold \textit{for all extended multiple zeta values}. This means that it is not important that multiple zeta functions are sums over natural numbers; instead, the important factor is that depth $k$ MZVs are sums over a certain simplex in $\mathbb{Z}^k$.

As a natural result of our study, we introduce the \textbf{complementary zeta function} $\tilde{\zeta}_{\zc}$, and stress both its importance and simplicity. Given a sequence of $N$ complex numbers $\zc := \{z_k\}$, where we allow $N \to \infty$, the complementary zeta function associated to this system is a zeta function built from the sequence of numbers  $\{\tilde{z}_{k}\}$ defined as
$$\frac{1}{\tilde{z}_{k}}=\sum_{i=1}^{k-1}\frac{1}{z_{i}-z_{k}}+\sum_{i=k+1}^{N}\left(\frac{1}{z_{i}-z_{k}}-\frac{1}{z_{i}}\right).$$
In the case $z_k=k, N \to \infty,$ we also have $\tilde{z}_k=k$ and the complementary zeta function reduces to the Riemann case. In certain cases, such as $z_k=k^2$, we can write $\tilde{\zeta}_{\zc}$ as a nonlinear combination of values of Riemann MZVs. We are not sure whether our theorems give nonlinear dependence relations which can be deduced from known shuffle and stuffle relations for MZVs. Systematically studying certain special cases of $\z$ may give new nontrivial linear dependence relations amongst MZVs of a given depth.

In Section \ref{compsec} we introduce the complementary zeta function, and in Section \ref{Section2} we show that the complementary zeta function naturally occurs when generalizing six identities from Riemann MZVs to extended MZVs. In Section \ref{casessec} we explore four natural specializations of the complementary zeta function. Finally, in Section \ref{section:Bessel} we study the Bessel zeta function, built from the finite sequence of zeros of the Bessel function. This work is the first in a natural program to extend identities for multiple zeta functions to identities in $QSym$. 


\section{Complementary zeta function}\label{compsec}
We will now carefully state and examine the definition of a complementary zeta function and its multiple analog, as well as several special cases. For any integer $N>1$ and $\zc = \left\{ z_{1},\dots,z_{N}\right\} $ a set of
nonzero complex numbers, define the \textbf{complementary sequence} $\left\{ \tilde{z}_{k}\right\} _{1\le k\le N}$
as 
\footnote{ it is assumed that in \eqref{eq:1/ztilde 1} a sum $\sum_{N}^{M}$ is equal to $0$ when $M<N.$ 
} 
\begin{align}
\frac{1}{\tilde{z}_{k}}&:=\sum_{i=1}^{k-1}\frac{1}{z_{i}-z_{k}}+\sum_{i=k+1}^{N} \left(\frac{1}{z_{i}-z_{k}}-\frac{1}{z_{i}} \right)
\label{eq:1/ztilde 1} \\
&=\sum_{i=1}^{k-1}\frac{1}{z_{i}-z_{k}}+\sum_{i=k+1}^{N} \frac{z_k}{z_i(z_i-z_k)} \nonumber.
\end{align}
Then the \textbf{complementary zeta function} is defined as the series  
\[
\tilde{\zeta}_{\zc}\left(s\right):=\sum_{n=1}^\infty\frac{1}{\tilde{z}_{n}z_{n}^{s-1}}.
\]
When $z_k=k$ with $N\to \infty$, we have $\tilde{z}_k=k$ since
\begin{align*}
\frac{1}{\tilde{z}_{k}}&=\sum_{i=1}^{k-1}\frac{1}{{i}-{k}}+\sum_{i=k+1}^{\infty}\left(\frac{1}{{i}-{k}}-\frac{1}{{i}}\right)= -\sum_{i=1}^{k-1}\frac{1}{{i}} + \sum_{i=1}^{k}\frac{1}{{i}}= \frac{1}{k},
\end{align*}
and hence $\z(s) = \zeta(s)$. The complementary zeta function is therefore seen as a nonstandard, but very natural generalization of the Riemann zeta function.  The surprise in the definition is that the complementary zeta function naturally arises when trying to generalize many outwardly different identities for Riemann MZVs to extended MZVs. For example, in Section \ref{sec3.1} the complementary zeta function naturally arises when considering the residues of rational functions. Meanwhile, in Section \ref{Section4} we treat a zeta sum formula due to Hirose-Murahara-Onozuka where the complementary zeta arises from elementary rational telescoping, a completely different technique. 

Due to the fact that we can state results relating only the extended and complementary zeta function (such as the prototypical Euler identity of Corollary \ref{eulercor}), after evaluating the complementary zeta function we obtain many identities ``for free," such as for the multiple $t$-values ($z_k = 2k+1$). However, finding closed forms for the complementary sequence $\tilde{z}_k$ and the complementary zeta function $\tilde{\zeta}_{\zc}$ is a formidable problem. 
Some special cases are given in Section \ref{casessec}.

We now introduce a multiple analog of the complementary zeta function; whether this is the most natural generalization remains to be seen. For all $N>1$ and 
$\{ z_n \}$ a sequence of 
nonzero complex numbers, define the \textbf{higher order complementary sequence} $\left\{ \tilde{z}_{n}^{\left(r\right)}\right\} $
as 
\begin{align}
\frac{1}{\tilde{z}_{n}^{\left(r\right)}} & :=\sum_{n>n_{2}>n_{3}>\dots>n_{r}}\frac{1}{\left(z_{n_{2}}-z_{n}\right)\dots\left(z_{n_{r}}-z_{n}\right)}
\label{eq:1/ztilde 1-1}\\
 &\qquad +\sum_{n_{1}>n>n_{3}>\dots>n_{r}}\frac{z_{n}}{z_{n_{1}}\left(z_{n_{1}}-z_{n}\right)\left(z_{n_{3}}-z_{n}\right)\dots\left(z_{n_{r}}-z_{n}\right)}\nonumber \\
 & \qquad+\dots\nonumber \\
 & \qquad+\sum_{n_{1}>n_{2}>\dots>n_{r-1}>n}\frac{z_{n}}{z_{n_{1}}\left(z_{n_{1}}-z_{n}\right)\left(z_{n_{2}}-z_{n}\right)\dots\left(z_{n_{r-1}}-z_{n}\right)}\nonumber \\
&=  \sum_{n>n_{2}>n_{3}>\dots>n_{r}}\frac{1}{\left(z_{n_{2}}-z_{n}\right)\dots\left(z_{n_{r}}-z_{n}\right)} \nonumber \\
&\qquad+ \sum_{ i=2 }^r   \sum_{n_1 > \cdots > n_{i-1}> n>n_{i+1}>\cdots > n_r } \frac{z_n}{z_{n_1}  \prod_{\substack{j=1 \\ j \neq i}}^r (z_{n_j}-z_n)} \nonumber.
\end{align}
When $z_k = k$, another telescoping argument gives $\tilde{z}_n^{(r)} = n^{r-1}$. The \textbf{higher order complementary zeta function} is then defined as
\begin{equation}
\tilde{\zeta}_{\mathcal{Z}}^{\left( r \right)}\left( k \right) := \sum_{n= 1}^\infty \frac{1}{\tilde{z}_{n}^{\left( r \right)}z_{n}^{k-1}}.\label{comp2}
\end{equation}

The discussion about lifting MZV identities to identities in $QSym$ from the introduction suggests an important open question: express $\tilde{\zeta}_{\mathcal{Z}}$ and $\tilde{\zeta}_{\mathcal{Z}}^{\left( r \right)}$ as quasisymmetric functions by expanding them in terms of a suitable basis of $QSym$, such as the elementary symmetric functions. Do these correspond to known symmetric functions?

\section{Structural Identities}\label{Section2}
In this section, we lift six identities from MZVs to extended MZVs. Note that the {complementary zeta function} and complementary sequence $\left\{ \tilde{z}_{k}\right\}$ appear in all six generalizations. The definition of $\left\{ \tilde{z}_{k}\right\}$ is exactly appropriate to allow many rational telescoping arguments to proceed; in essence, we reduce difficult telescoping arguments involving extended MZVs to the evaluation of the complementary zeta function.

\subsection{A somewhat unlikely looking identity}\label{sec3.1}
In their legendary paper \cite{BorweinGoldbach}, Borwein and Bradley provide 32 different proofs of  
Euler's identity
\begin{equation}
\zeta\left(2,1\right)=\zeta\left(3\right).
\label{Euler identity}
\end{equation}
One of their proofs involves the ``somewhat unlikely
looking'' identity 
\begin{equation}
\sum_{n\ge1}\frac{1}{n\left(n+x\right)}\sum_{m=1}^{n-1}\frac{1}{m+x}=\sum_{n\ge1}\frac{1}{n^{2}\left(n+x\right)},
\label{eq:main}
\end{equation}
that holds for any non negative integer $x,$ and is obtained in \cite {Borwein} by the elementary manipulation
of rational fractions.
The interesting feature of this identity is that choosing $x=0$ provides
Euler's identity \eqref{Euler identity}.
As we will see, this identity is at the root of many structural identities for MZVs.
We first give an extended structural version of this result.
\begin{theorem}
\label{thm1}
Consider a sequence $\{z_k\}_{1 \le k \le N}$ of non-zero complex numbers. Then, with $\tilde{z}_{n}$ defined by \eqref{eq:1/ztilde 1}, we have
\begin{equation}
\sum_{n=1}^{N}\frac{1}{z_{n}\left(z_{n}+x\right)}\sum_{m=1}^{n-1}\frac{1}{z_{m}+x}=\sum_{n=1}^{N}\frac{1}{z_{n}\tilde{z}_{n}}\frac{1}{z_{n}+x}.\label{eq:N finite}
\end{equation}
\end{theorem}

\begin{corollary}\label{eulercor}
The special case $x=0$ and $N\to \infty$ in \eqref{eq:N finite} gives,
with
\[
\tilde{\zeta}_{\zc}\left(3\right)=\sum_{n=1}^\infty\frac{1}{z_{n}^{2}\tilde{z}_{n}},
\]
the following generalization of Euler's identity:
\begin{equation}
\z\left(2,1\right)=\tilde{\zeta}_{\zc}\left(3\right).
\label{Euler generalized}
\end{equation}
\end{corollary}

\begin{proof}[Proof of Theorem \ref{thm1}]
First consider the case when $N$ is finite. Both sides of (\ref{eq:N finite})
are rational functions; let us show that they have the same poles
and same residue at each pole. Clearly, the left-hand side has only
simple poles $-z_{k}$; for each pole $-z_{k},$ the residue $\alpha_{k}$
is computed as
\[
\alpha_{k}=\lim_{x\to-z_{k}}\left(x+z_{k}\right)\left(\sum_{n=1}^{N}\frac{1}{z_{n}\left(z_{n}+x\right)}\sum_{m=1}^{n-1}\frac{1}{z_{m}+x}\right).
\]
Assuming first $k>1,$ we have
\begin{align*}
\alpha_k&=\lim_{x\to-z_{k}}\left(x+z_{k}\right)\left(\sum_{n=1}^{k-1}\frac{1}{z_{n}\left(z_{n}+x\right)}\sum_{m=1}^{n-1}\frac{1}{z_{m}+x}\right) + \lim_{x\to-z_{k}} \frac{1}{z_k} \sum_{m=1}^{k-1} \frac{1}{z_m+x}  \\
&\qquad+\lim_{x\to-z_{k}}\left(x+z_{k}\right)\left(\sum_{n=k+1}^{N}\frac{1}{z_{n}\left(z_{n}+x\right)}\sum_{m=1}^{n-1}\frac{1}{z_{m}+x}\right)\\
 & =0+\frac{1}{z_{k}}\left\{ \frac{1}{z_{1}-z_{k}}+\dots+\frac{1}{z_{k-1}-z_{k}}\right\} +\frac{1}{z_{k+1}\left(z_{k+1}-z_{k}\right)}+\dots+\frac{1}{z_{N}\left(z_{N}-z_{k}\right)} \\
&= \frac{1}{z_k \tilde{z}_k}
\end{align*}
after comparing to the definition of $\frac{1}{\tilde{z}_k}$ in (\ref{eq:1/ztilde 1}). The computation
of the boundary residues $\alpha_{1}$ and $\alpha_{N}$ is equally
simple. 
\end{proof}
In \cite{Borwein}, the authors show that computing the $s-$th Taylor coefficient
in (\ref{eq:main}) yields a corresponding  formula for the Riemann zeta value $\zeta(s)$:
\begin{equation}
\label{sum rule}
\zeta\left(s+3\right)=\sum_{\substack{a+b=s \\ a,b\ge0}}\zeta\left(2+a,1+b\right).
\end{equation}
For example, with $s=1$, this yields
\[
\zeta\left(4\right)=\zeta\left(3,1\right)+\zeta\left(2,2\right).
\]
The same approach, i.e. taking $N\to \infty$ and computing the $s-$th Taylor coefficient in (\ref{eq:N finite}),
provides in our case the following extension of the  formula \eqref{sum rule}.
\begin{theorem}
For
any sequence $\left\{ z_{k}\right\} _{k\ge1}$ of nonzero complex
numbers such that the following sums are convergent,
\begin{equation}
\label{zetak+3}
\tilde{\zeta}_{\zc}\left(s+3\right)
=
\sum_{\substack{a+b=s \\ a,b\ge0}}\z\left(2+a,1+b\right)
.
\end{equation}
\end{theorem}

\begin{proof}
Since 
\[
\frac{d^{k}}{dx^{k}}\sum_{n>m\ge1}\frac{1}{z_{n}\left(z_{n}+x\right)}\frac{1}{z_{m}+x}=\frac{d^{k}}{dx^{k}}\sum_{n\ge1}\frac{1}{z_{n}\tilde{z}_{n}\left(z_{n}+x\right)},
\]
applying  Leibniz' formula
\[
\frac{d^{k}}{dx^{k}}\frac{1}{z_{n}+x}\frac{1}{z_{m}+x}=\sum_{j=0}^{k}\binom{k}{j}\left(\frac{d^{j}}{dx^{j}}\frac{1}{z_{n}+x} \right)\left(\frac{d^{k-j}}{dx^{k-j}}\frac{1}{z_{m}+x}\right)
\]
and
\[
\frac{d^{j}}{dx^{j}}\frac{1}{z_{n}+x}=\frac{\left(-1\right)^{j}j!}{\left(z_{n}+x\right)^{j+1}},
\]
we deduce
\[
\sum_{j=0}^{k}\sum_{n>m\ge1}\frac{1}{z_{n}\left(z_{n}+x\right)^{j+1}}\frac{1}{\left(z_{m}+x\right)^{k-j+1}}=\sum_{n\ge1}\frac{1}{z_{n}\tilde{z}_{n}\left(z_{n}+x\right)^{k+1}}.
\]
Evaluating this at $x=0$ yields
\[
\sum_{j=0}^{k}\sum_{n>m\ge1}\frac{1}{z_{n}^{j+2}}\frac{1}{z_{m}^{k-j+1}}=\sum_{n\ge1}\frac{1}{\tilde{z}_{n}z_{n}^{k+2}},
\]
which can be rewritten as
\[
\sum_{j=0}^{k}\z\left(j+2,k-j+1\right)=\tilde{\zeta}_{\zc}\left(k+3\right),
\]
which is the desired result.
\end{proof}

\subsection{Sum formula}
The sum formula for the ordinary MZVs is the following identity
\begin{equation}
\label{sum formula}
\sum_{\substack{\sum_i a_i=s \\ a_i\ge0}}
\zeta\left( a_1+2,a_2+1,\dots,a_r+1 \right)
=
\zeta\left( r+s+1 \right).
\end{equation}
As noticed in \cite{Borwein}, this formula can be derived by identifying the $s-$th Taylor coefficient in the generalization
\[
\sum_{k_1>k_2>\dots>k_r>0}\frac{1}{k_1}\prod_{j=1}^{r}\frac{1}{k_j-x} = \sum_{n=1}^{\infty} \frac{1}{n^r\left( n-x \right)}
\]
of \eqref{eq:main}.
An extension of our previous methods allows us to state a structural version of the sum formula. We first derive the general case of Theorem \ref{thm1} as follows.
\begin{theorem}
\label{thm general ztilde_r}
With $\tilde{z}_{n}^{\left(r\right)}$ defined in \eqref{eq:1/ztilde 1-1}, we have
\begin{equation}
\sum_{n_{1}>\dots>n_{r}\ge1}\frac{1}{z_{n_{1}}\left(z_{n_{1}}+x\right)\dots\left(z_{n_{r}}+x\right)}=\sum_{n\ge1}\frac{1}{z_{n}\tilde{z}_{n}^{\left(r\right)}}\frac{1}{z_{n}+x}.\label{eq:N finite-1}
\end{equation}
\end{theorem}

\begin{proof}
The right-hand side of (\ref{eq:N finite-1}) has poles at $x=-z_{n},\thinspace\thinspace$
with residues $\frac{1}{z_{n}\tilde{z}_{n}^{\left( r \right)}}$. 
The residue at $x=-z_n$ on the left-hand side can be computed as
\begin{align*}
&\lim_{x\to-z_{n}}\left( x+z_n \right)\sum_{n_{1}>\dots>n_{r}\ge1}\frac{1}{z_{n_{1}}\left(z_{n_{1}}+x\right)\dots\left(z_{n_{r}}+x\right)} \\
&=\sum_{n>n_{2}>n_{3}>\dots>n_{r}}\frac{1}{z_n\left(z_{n_{2}}-z_{n}\right)\dots\left(z_{n_{r}}-z_{n}\right)} \\
&\qquad+\sum_{n_{1}>n>n_{3}>\dots>n_{r}}\frac{z_{n}}{z_{n_1}\left(z_{n_{1}}-z_{n}\right)\left(z_{n_{3}}-z_{n}\right)\dots\left(z_{n_{r}}-z_{n}\right)}\\ 
&\qquad+\dots +\sum_{n_{1}>n_{2}>\dots>n_{r-1}>n}\frac{z_{n}}{z_{n_{1}}\left(z_{n_{1}}-z_{n}\right)\left(z_{n_{2}}-z_{n}\right)\dots\left(z_{n_{r-1}}-z_{n}\right)} \\
&=\frac{1}{z_{n}\tilde{z}_{n}^{\left( r \right)}},
\end{align*}
 after comparing with Definition (\ref{eq:1/ztilde 1-1}). We obtain $r$ multiple summations based on the case for which of the summation indices equals the fixed $n$.
\end{proof}

As a consequence of Theorem \ref{thm general ztilde_r}, we deduce the following generalization of the sum formula \eqref{sum formula}.
\begin{theorem}
With $\{\tilde{z}_{n}^{\left( r \right)}\}$ defined as in \eqref{eq:1/ztilde 1-1} and $\tilde{\zeta}_\zc$ defined by \eqref{comp2}, the extended multiple zeta function satisfies the sum rule
\[
\sum_{\substack{\sum s_i =s\\s_i\ge0}} \z\left( s_1+2,s_2+1,\dots,s_r+1 \right) = \tilde{\zeta}_{\zc}^{\left( r \right)}\left( r+s+1 \right).
\]
\end{theorem}
\begin{proof}
Compute the Taylor expansion of each side of \eqref{eq:N finite-1} and identify the coefficient of $x^{r+s+1}$ on both sides.
\end{proof}
\subsection{Euler's reduction formula}
In \cite{Borwein}, identity \eqref{sum rule} is referred to as an {\it inversion of Euler's reduction formula,} that reads 
\[
\zeta\left( s,1 \right) = \frac{s}{2}\zeta\left( s+1 \right)-\frac{1}{2}\sum_{k=1}^{s-2}\zeta\left( k+1 \right)\zeta\left( s-k \right),\,\,s>1 \in \mathbb{Z}.
\]
The next result provides an extension of this reduction formula to the extended MZVs, showing that it is of structural type.

\begin{theorem}
The  generalized reduction formula is
\[
\z\left( s,1 \right) = 
\tilde{\zeta}_\zc\left( s+1 \right)
+
\left( \frac{s}{2}-1 \right)\z\left( s+1 \right)
-\frac{1}{2}\sum_{k=1}^{s-2}\z\left( k+1 \right)\z\left( s-k \right).
\]
\end{theorem}
\begin{proof}
Start from \eqref{zetak+3} with $k+2$ replaced by $s$ and extract the last term in the sum
to obtain 
\[
\z\left(s,1\right)=
\tilde{\zeta}_\zc\left(s+1\right)
-\sum_{j=1}^{s-2}\z\left(j+1,s-j\right).
\]
Next apply the reflection formula \eqref{reflection} to each couple of terms $\z\left(j+1,k-j\right)$ and $\z\left(k-j,j+1\right)$ in the sum so that
\begin{align*}
\z\left(s,1\right)&=
\tilde{\zeta}_\zc\left(s+1\right)
-\frac{1}{2}\sum_{j=1}^{s-2}
\left[ \z\left( j+1 \right)\z\left( s-j \right) -\z\left( s+1 \right)\right]\\
&=\tilde{\zeta}_\zc\left(s+1\right)
+\frac{s-2}{2}\z\left( s+1 \right)
-\frac{1}{2}\sum_{j=1}^{s-2}
\z\left( j+1 \right)\z\left( s-j \right), 
\end{align*}
which is the desired result.
\end{proof}
\subsection{The sum formula by Hirose et al.}\label{Section4}
In \cite{Hirose}, Hirose, Murahara, and Onozuka propose a generalization of the depth $2$ case of the sum formula \eqref{sum formula} to arbitrary complex values of the parameter $s$ as follows:
\begin{theorem}
\cite[Theorem 1.2]{Hirose} for $\Re\left(s\right)>1,$
\begin{equation}
\sum_{n\ge0}\left(\zeta\left(n+2,s-n-2\right)-\zeta\left(s+n,-n\right)\right)=\zeta\left(s\right).\label{eq:Hirose}
\end{equation}
\end{theorem}

We prove a further generalization:
\begin{theorem}
We have, for complex $s$ such that the following quantities exist,
\begin{equation}
\sum_{n\ge0}\left(\z\left(n+2,s-n-2\right)-\z\left(s+n,-n\right)\right)=\tilde{\zeta}_{\zc}\left(s\right).\label{eq:Hirose generalized}
\end{equation}
\end{theorem}

\begin{proof}
The left-hand side is expanded as
\begin{align*}
\sum_{n\ge0}\sum_{0<n_{2}<n_{1}}\frac{1}{z_{n_{1}}^{n+2}z_{n_{2}}^{s-n-2}}-\frac{1}{z_{n_{1}}^{s+n}z_{n_{2}}^{-n}} & =\sum_{0<n_{2}<n_{1}}\sum_{n\ge0}\left(\frac{z_{n_{2}}}{z_{n_{1}}}\right)^{n}\left(\frac{1}{z_{n_{1}}^{2}z_{n_{2}}^{s-2}}-\frac{1}{z_{n_{1}}^{s}}\right)\\
 & =\sum_{0<n_{2}<n_{1}}\frac{1}{1-\frac{z_{n_{2}}}{z_{n_{1}}}}\left(\frac{1}{z_{n_{1}}^{2}z_{n_{2}}^{s-2}}-\frac{1}{z_{n_{1}}^{s}}\right)\\
 & =\sum_{0<n_{2}<n_{1}}\frac{1}{z_{n_{1}}^{2}-z_{n_{1}}z_{n_{2}}}\left(\frac{1}{z_{n_{2}}^{s-2}}-\frac{1}{z_{n_{1}}^{s-2}}\right)\\
 & =S_{1}-S_{2},
\end{align*}
with
\begin{align*}
S_{1} & =\sum_{0<n_{2}<n_{1}}\frac{1}{z_{n_{1}}^{2}-z_{n_{1}}z_{n_{2}}}\frac{1}{z_{n_{2}}^{s-2}}=\sum_{0<n_{2}}\frac{1}{z_{n_{2}}^{s-2}}\sum_{n_{1}>n_{2}}\frac{1}{z_{n_{1}}^{2}-z_{n_{1}}z_{n_{2}}}\\
 & =\sum_{0<n_{2}}\frac{1}{z_{n_{2}}^{s-2}}\sum_{n_{1}>n_{2}}\frac{1}{z_{n_{2}}}\left(\frac{1}{z_{n_{1}}-z_{n_{2}}}-\frac{1}{z_{n_{1}}}\right)\\
 & =\sum_{0<n_{2}}\frac{1}{z_{n_{2}}^{s-1}}\sum_{n_{1}>n_{2}}\left(\frac{1}{z_{n_{1}}-z_{n_{2}}}-\frac{1}{z_{n_{1}}}\right)
\end{align*}
and
\[
S_{2}=\sum_{0<n_{2}<n_{1}}\frac{1}{z_{n_{1}}^{2}-z_{n_{1}}z_{n_{2}}}\frac{1}{z_{n_{1}}^{s-2}}=\sum_{n_{1}>1}\frac{1}{z_{n_{1}}^{s-1}}\sum_{0<n_{2}<n_{1}}\frac{1}{z_{n_{1}}-z_{n_{2}}}.
\]
We deduce
\begin{align*}
&\sum_{n\ge0}\left[\z\left(n+2,s-n-2\right)-\z\left(s+n,-n\right)\right] \\
& =S_{1}-S_{2}\\
&=\sum_{0<n_{2}}\frac{1}{z_{n_{2}}^{s-1}}\sum_{n_{1}>n_{2}}\left(\frac{1}{z_{n_{1}}-z_{n_{2}}}-\frac{1}{z_{n_{1}}}\right) -\sum_{n_{1}>1}\frac{1}{z_{n_{1}}^{s-1}}\sum_{0<n_{2}<n_{1}}\frac{1}{z_{n_{1}}-z_{n_{2}}}.
\end{align*}

Reindexing, this is
\[
\sum_{k>0}\frac{1}{z_{k}^{s-1}}\left[\sum_{i>k}\left(\frac{1}{z_{i}-z_{k}}-\frac{1}{z_{i}}\right)+\sum_{0<i<k}\frac{1}{z_{i}-z_{k}}\right]=\sum_{k>0}\frac{1}{z_{k}^{s-1}\tilde{z}_{k}}.
\]
\end{proof}


\section{The Complementary Zeta Function: Rational Cases}\label{casessec}

In this section, we study several special cases of the complementary multiple zeta function, which correspond to some specific choices for the  sequence $\{z_{k}\}$.
In what follows, the \textbf{Hurwitz zeta function} is denoted as
\[
\zeta_{H}\left(s,z\right)=\sum_{n\ge0}\frac{1}{\left(z+n\right)^{s}},
\]
in order to avoid  the confusion with the MZV of depth 2,
\[
\zeta\left(a,b\right)=\sum_{n>m>0}\frac{1}{n^{a}m^{b}}.
\]
Furthermore, we will let $$\psi(z):= \frac{\Gamma'(z)}{\Gamma(z)} = -\gamma + \sum_{n=1}^\infty \left(\frac{1}{n} - \frac{1}{n+z-1}\right) $$ denote the \textit{digamma function}, where $\gamma \simeq 0.57722$ is the Euler-Mascheroni constant.
We will also frequently refer to the \textit{multiple $t$-values} of Hoffman \cite{Hoffman}, defined as
\[
t\left(s_{1},\dots,s_{k}\right)=\sum_{\substack{n_{1}>\dots>n_{k}\ge1 \\ n_{i}\thinspace\thinspace\text{odd\thinspace\thinspace}}}\frac{1}{n_{1}^{s_{1}}\dots n_{k}^{s_{k}}}
\]
with the special case
$t\left(s\right)=\left(1-2^{-s}\right)\zeta\left(s\right).$ Note that, up to a power of $2$, this is equivalent to the extended multiple zeta function with $z_k = 2k+1$.

We first collect four evaluations of the complementary zeta function, where $z_k$ is a simple linear or quadratic term, and defer their proofs to the end of this section.
\begin{theorem}
Consider $a \in [0,1]$ and let $z_k = k + a - 1$. Then
$$ \frac{1}{\tilde{z}_{k}} = \psi\left(k+a\right)-\psi\left(k\right),$$
and 
$$\tilde{\zeta}_{\zc}\left(s\right)= \zeta_{H}\left(s-1,a\right)\left[\psi\left(a+1\right)-\psi\left(1\right)\right]-a\sum_{l=0}^{\infty}\frac{\zeta_{H}\left(s-1,a+l+1\right)}{\left(a+l+1\right)\left(l+1\right)}.$$
\end{theorem}

\begin{theorem}
For $s \geq 2$ an integer and $z_k=k-\frac{1}{2}$, the complementary MZV $\tilde{\zeta}_{\zc}\left(s\right)$ is
\begin{align}
\tilde{\zeta}_{\zc}\left(s\right) & 
=2^{s}\left[t\left(s\right)+t\left(s-1,1\right)+\frac{\psi\left(\frac{1}{2}\right)}{2}t\left(s-1\right)\right]\nonumber \\
 & -\frac{\left(-1\right)^{s-1}}{\left(s-2\right)!}\left[\sum_{k=0}^{s-3}\binom{s-3}{k}\psi^{\left(k+1\right)}\left(\frac{1}{2}\right)\psi^{\left(s-k-3\right)}\left(\frac{1}{2}\right)-\frac{1}{2}\psi^{\left(s-1\right)}\left(\frac{1}{2}\right)\right].\label{eq:zetatildeG(s)}
\end{align}
At odd positive integers ($s\geq 1$ in what follows), we can also eliminate the multiple $t$-values and evaluate
\begin{align*}
\tilde{\zeta}_{\zc}\left(2s+1\right) & =-\gamma\zeta\left(2s\right)\left(2^{2s}-1\right)+\left(s+\frac{1}{2}\right)\zeta\left(2s+1\right)\left(2^{2s+1}-1\right)\\
 & -\sum_{l=1}^{s-1}\left(2^{2l}-1\right)\zeta\left(2l\right)\zeta\left(2s+1-2l\right)\\
 & +\frac{1}{2s-1}\left[\sum_{k=0}^{2s-2}\left(k+1\right)\left(2^{k+2}-1\right)\zeta\left(k+2\right)\left(2^{2s-k-1}-1\right)\zeta\left(2s-k-1\right)\right].
\end{align*}
\end{theorem}

\begin{theorem}
In the case $z_{k}=k^{2},$  we have
\[
\frac{1}{\tilde{z}_{k}}=\frac{3}{4k^{2}}-\psi'\left(k+1\right).
\]
As a consequence, the complementary zeta function is 
\[
\tilde{\zeta}_{\zc}\left(s\right)=\frac{7}{4}\zeta\left(2s\right)-\zeta\left(2\right)\zeta\left(2s-2\right)+\zeta\left(2s-2,2\right).
\]
\end{theorem}
\begin{theorem}
In the case $z_{k}=k\left(k+1\right),$ we have
\[
\frac{1}{\tilde{z}_{k}}=\frac{1}{k}-\frac{2}{k+1}+\frac{1}{\left(2k+1\right)^{2}},
\]
and the corresponding complementary MZV is a linear combination of values of the Riemann zeta function
\[
\tilde{\zeta}_{\zc}\left(s\right)=\left(-1\right)^{s}\left(\sum_{k=2}^{s}\mu_{k}^{\left(s\right)}\zeta\left(k\right)+\eta_{s}\right).
\]
Furthermore, we have
\[
\mu_{k}^{\left(s\right)}=\left(\left(-1\right)^{k}+2\right)\left\{ \binom{2s-2-k}{s-2}+\left(-1\right)^{k}\binom{2s-2-k}{s-1}\right\} -\left(1+\left(-1\right)^{k}\right)\beta_{k-1}^{\left(s-1\right)}
\]
with 
\[
\beta_{k}^{\left(s\right)}=\sum_{i=0}^{s-k-1}4^{i}\binom{2s-2i-k-2}{s-i-1}
\]
and
\[
\eta_{s}=\left(s-\frac{1}{2}\right)\binom{2s-2}{s-1}-\binom{2s-2}{s}-4\binom{2s-3}{s-2}-\left(\frac{\pi^{2}}{8}-\frac{1}{2}\right)2^{2s-2}.
\]
\end{theorem}

\begin{remark}
For $z_k = k (k+1)$, the first values are
\[
\tilde{\zeta}_{\zc}\left(2\right)=0,\thinspace\thinspace\tilde{\zeta}_{\zc}\left(3\right)=-7+\frac{5\pi^{2}}{6}-\zeta\left(3\right),\thinspace\thinspace\tilde{\zeta}_{\zc}\left(4\right)=47-\frac{16\pi^{2}}{3}+\frac{\pi^{4}}{30}+2\zeta\left(3\right).
\]
Note that the values of $\mu_k^{\left( s \right)}$ for odd $k$, i.e. the coefficients of $\zeta\left( 2k+1 \right)$, simplify to
\[
\mu_{2k+1}^{\left( s \right)} = \frac{2k}{s-1} \binom{2s-2k-1}{s-2}.
\]
Moreover, for arbitrary $k,$ the sequence of coefficients $\beta_k^{\left( s \right)}$ coincides with the OEIS  sequence A143019 \cite{OEIS}
read by antidiagonals: more precisely, denote as $a^{\left( q \right)}_{n}$ the $n-$th coefficient in the Taylor series expansion at $z=0$ of the function 
\[
\frac{1}{\left( 1-4z \right)^{\frac{3}{2}}}\left( \frac{1-\sqrt{1-4z}}{2z} \right)^q,
\]
then 
\[
\beta_{k}^{\left( s \right)} = a^{\left( k \right)}_{s-k-1}.
\]
As a consequence of the recurrence identity $a^{\left( q \right)}_{n} = a^{\left( q-1 \right)}_{n}  + a^{\left( q+1 \right)}_{n-1},$ the coefficients $\beta_{k}^{\left( s \right)}$ satisfy the three term recurrence
\[
\beta_{k}^{\left( s+k+1 \right)}-\beta_{k+1}^{\left( s+k+1 \right)} = \beta_{k-1}^{\left( s+k \right)}.
\]
\end{remark}

We now give proofs of all four cases.
\begin{proof}[Proof of Theorem 4.1]
The choice $z_{k}=k+a-1$ gives
\[
\frac{1}{\tilde{z}_{k}}=\sum_{i=1}^{k-1}\frac{1}{i-k}+\sum_{i=k+1}^{\infty}\frac{1}{i-k}-\frac{1}{i+a-1}=\psi\left(k+a\right)-\psi\left(k\right).
\]
Note that for $a=1$ we recover the Riemann case
\[
\frac{1}{\tilde{z}_{k}}=\frac{1}{k}.
\]
The corresponding zeta function is
\begin{equation}\label{linzeta}
\tilde{\zeta}_{\zc}\left(s\right)=\sum_{n=1}^\infty\frac{1}{\tilde{z}_{n}z_{n}^{s-1}}=\sum_{n\ge1}\frac{\psi\left(n+a\right)-\psi\left(n\right)}{\left(n+a-1\right)^{s-1}}
\end{equation}
and can be computed as follows:
\begin{align*}
\tilde{\zeta}_{\zc}\left(s\right) & =\zeta_{H}\left(s-1,a\right)\left[\psi\left(a+1\right)-\psi\left(1\right)\right]-a\sum_{l=0}^{\infty}\frac{\zeta_{H}\left(s-1,a+l+1\right)}{\left(a+l+1\right)\left(l+1\right)}.
\end{align*}
To simplify this, we first use \cite[(B.6)]{Milgram}:
\begin{align}
\sum_{l=0}^{k}\frac{\psi\left(b+l\right)}{\left(c+l\right)^{s}} & =\psi\left(b\right)\zeta_{H}\left(s,c\right)-\psi\left(b+k+1\right)\zeta_{H}\left(s,c+k+1\right)\label{eq:Milgram}\\
 & \qquad+\sum_{l=0}^{k}\frac{\zeta_{H}\left(s,c+l+1\right)}{b+l}.\nonumber 
\end{align}
Choosing $b=a+1$ and $c=a,$ and next $b=1$ and $c=a,$ and substituting
$s$ with $s-1$ yields 
\begin{align*}
\sum_{n=1}^{k+1}\frac{\psi\left(n+a\right)-\psi\left(n\right)}{\left(n+a-1\right)^{s-1}} & =\zeta_{H}\left(s-1,a\right)\left[\psi\left(a+1\right)-\psi\left(1\right)\right]\\
 & \qquad-\zeta_{H}\left(s-1,a+k+1\right)\left[\psi\left(a+k+2\right)-\psi\left(k+2\right)\right]\\
 & \qquad+\sum_{l=0}^{k}\zeta_{H}\left(s-1,a+l+1\right)\left[\frac{1}{a+l+1}-\frac{1}{l+1}\right].
\end{align*}
Taking the limit $k\to\infty$ gives the desired result.
\end{proof}

%
%

\begin{proof}[Proof of Theorem 4.2]
On setting $a=\frac{1}{2}$ in \eqref{linzeta}, the MZV simplifies to
\[
\tilde{\zeta}_{\zc}\left(s\right)=\sum_{n\ge1}\frac{\psi\left(n+\frac{1}{2}\right)-\psi\left(n\right)}{\left(n-\frac{1}{2}\right)^{s-1}}.
\]
Notice that
\[
\psi\left(n+\frac{1}{2}\right)-\psi\left(n\right)=2\sum_{j=0}^{\infty}\frac{\left(-1\right)^{j}}{2n+j}=\sum_{j=0}^{\infty}\frac{1}{n+j}-\sum_{j=0}^{\infty}\frac{1}{n+j+\frac{1}{2}},
\]
so that 
\[
\tilde{\zeta}_{\zc}\left(s\right)=\sum_{n=1}^\infty\sum_{j=0}^\infty\frac{1}{\left(n+j\right)\left(n-\frac{1}{2}\right)^{s-1}}-\sum_{n=1}^\infty\sum_{j=0}^\infty\frac{1}{\left(n+j+\frac{1}{2}\right)\left(n-\frac{1}{2}\right)^{s-1}}.
\]

The first part of the sum is computed as
\begin{align*}
\sum_{n=1}^\infty\frac{\psi\left(n+\frac{1}{2}\right)}{\left(n-\frac{1}{2}\right)^{s-1}} & =\sum_{n=1}^\infty\frac{1}{\left(n-\frac{1}{2}\right)^{s}}+\sum_{n=1}^\infty\frac{\psi\left(n-\frac{1}{2}\right)}{\left(n-\frac{1}{2}\right)^{s-1}}\\
 & =2^{s}t\left(s\right)+2^{s-1}\psi\left(\frac{1}{2}\right)+\sum_{n=2}^\infty\frac{1}{\left(n-\frac{1}{2}\right)^{s-1}}\left(\sum_{k=1}^{n-1}\frac{1}{k-\frac{1}{2}}+\psi\left(\frac{1}{2}\right)\right)\\
 & =2^{s}t\left(s\right)+2^{s}t\left(s-1,1\right)+2^{s-1}\psi\left(\frac{1}{2}\right)\\
&\qquad+\psi\left(\frac{1}{2}\right)\left(2^{s-1}t\left(s-1\right)-2^{s-1}\right)\\
 & =2^{s}t\left(s\right)+2^{s}t\left(s-1,1\right)+2^{s-1}t\left(s-1\right)\psi\left(\frac{1}{2}\right).
\end{align*}
The second term is computed using \cite[Theorem 4.10]{Milgram}, 
\[
\sum_{n=1}^\infty\frac{\psi\left(n\right)}{\left(n+q-1\right)^{s}}=\frac{\left(-1\right)^{s}}{\left(s-1\right)!}\left[\sum_{k=0}^{s-2}\binom{s-2}{k}\psi^{\left(k+1\right)}\left(q\right)\psi^{\left(s-k-2\right)}\left(q\right)-\frac{1}{2}\psi^{\left(s\right)}\left(q\right)\right].
\]
This gives Equation \eqref{eq:zetatildeG(s)}. For the specialization at odd integers, we use \cite[(C.28)]{Milgram} 
\begin{align*}
\sum_{n=1}^\infty\frac{\psi\left(n-\frac{1}{2}\right)}{\left(n-\frac{1}{2}\right)^{2s}} & =-\gamma\zeta\left(2s\right)\left(2^{2s}-1\right)-\frac{1}{2}\zeta\left(2s+1\right)\left(2^{2s+1}-1\right)\\
 & \qquad-\sum_{l=1}^{s-1}\left(2^{2l}-1\right)\zeta\left(2l\right)\zeta\left(2s+1-2l\right).
\end{align*}
Since moreover
\[
\psi^{\left(n\right)}\left(\frac{1}{2}\right)=\left(-1\right)^{n+1}n!\left(2^{n+1}-1\right)\zeta\left(n+1\right),
\]
we deduce
\[
\tilde{\zeta}_{\zc}\left(2s+1\right)=\sum_{n=1}^\infty\frac{\psi\left(n+\frac{1}{2}\right)-\psi\left(n\right)}{\left(n-\frac{1}{2}\right)^{2s}}=\sum_{n=1}^\infty\frac{\psi\left(n-\frac{1}{2}\right)-\psi\left(n\right)}{\left(n-\frac{1}{2}\right)^{2s}}+\sum_{n=1}^\infty\frac{1}{\left(n-\frac{1}{2}\right)^{2s+1}}.
\]
The last sum is 
\[
\sum_{n=1}^\infty\frac{1}{\left(n-\frac{1}{2}\right)^{2s+1}}=\left(2^{2s+1}-1\right)\zeta\left(2s+1\right),
\]
while the first sum is
\begin{align*}
\sum_{n=1}^\infty&\frac{\psi\left(n-\frac{1}{2}\right)-\psi\left(n\right)}{\left(n-\frac{1}{2}\right)^{2s}}  =-\gamma\zeta\left(2s\right)\left(2^{2s}-1\right)-\frac{1}{2}\zeta\left(2s+1\right)\left(2^{2s+1}-1\right)\\
 & \qquad-\sum_{l=1}^{s-1}\left(2^{2l}-1\right)\zeta\left(2l\right)\zeta\left(2s+1-2l\right)\\
 &\qquad -\frac{1}{\left(2s-1\right)!}\left[\sum_{k=0}^{2s-2}\binom{2s-2}{k}\psi^{\left(k+1\right)}\left(\frac{1}{2}\right)\psi^{\left(2s-k-2\right)}\left(\frac{1}{2}\right)-\frac{1}{2}\psi^{\left(2s\right)}\left(\frac{1}{2}\right)\right].
\end{align*}
We deduce
\begin{align*}
\tilde{\zeta}_{\zc}\left(2s+1\right) & =-\gamma\zeta\left(2s\right)\left(2^{2s}-1\right)+\left(s+\frac{1}{2}\right)\zeta\left(2s+1\right)\left(2^{2s+1}-1\right)\\
 & 
+\frac{1}{2s-1}\left[\sum_{k=0}^{2s-2}\left(k+1\right)\left(2^{k+2}-1\right)\zeta\left(k+2\right)\left(2^{2s-k-1}-1\right)\zeta\left(2s-k-1\right)\right]\\
& 
-\sum_{l=1}^{s-1}\left(2^{2l}-1\right)\zeta\left(2l\right)\zeta\left(2s+1-2l\right).
\end{align*}
\end{proof}

\begin{proof}[Proof of Theorem 4.3]
The computation of $\frac{1}{\tilde{z}_{k}}$ is straightforward. Moreover,
\[
\tilde{\zeta}_{\zc}\left(s\right)=\sum_{n=1}^\infty\left(\frac{3}{4n^{2}}-\psi^{'}\left(n+1\right)\right)\frac{1}{n^{2s-2}}=\frac{3}{4}\zeta\left(2s\right)-\sum_{n=1}^\infty\frac{\psi'\left(n+1\right)}{n^{2s-2}}.
\]
The sum is computed using \cite[(B.6)]{Milgram}: with
\begin{align*}
\sum_{l=0}^{k}\frac{\psi'\left(l+2\right)}{\left(l+1\right)^{2s-2}}&=\psi'\left(2\right)\zeta_{H}\left(2s-2,1\right)-\psi'\left(2+k+1\right)\zeta_{H}\left(2s-2,k+2\right) \\
&\qquad-\sum_{l=0}^{k}\frac{\zeta_{H}\left(2s-2,l+2\right)}{\left(l+2\right)^{2}},
\end{align*}
taking the limit $k\to\infty$ gives 
\[
\tilde{\zeta}_{\zc}\left(s\right)=\frac{3}{4}\zeta\left(2s\right)-\psi'\left(2\right)\zeta_{H}\left(2s-2,1\right)+\sum_{l=0}^{\infty}\frac{\zeta_{H}\left(2s-2,l+2\right)}{\left(l+2\right)^{2}}.
\]
The sum in the right-hand side is now computed as follows:
\begin{align*}
\sum_{l=0}^{\infty}\frac{\zeta_{H}\left(2s-2,l+2\right)}{\left(l+2\right)^{2}} & =\sum_{l=0}^\infty \sum_{n=0}^\infty\frac{1}{\left(l+2\right)^{2}\left(n+l+2\right)^{2s-2}}\\
&=\sum_{l = 1}^\infty\sum_{n=0}^\infty\frac{1}{\left(l+1\right)^{2}\left(n+l+1\right)^{2s-2}}\\
 & =\sum_{l=1}^{\infty}\frac{1}{\left(l+1\right)^{2s}}+\sum_{l=2}^{\infty}\sum_{n = 1}^\infty\frac{1}{l^{2}\left(n+l\right)^{2s-2}}.
\end{align*}
The first sum is $\zeta\left(2s\right)-1$
while the second sum is recognized as
\[
\zeta\left(2s-2,2\right) - \sum_{n=1}^{\infty}\frac{1}{1^2\left(1+n\right)^{2s-2}} = \zeta\left(2s-2,2\right) -\left( \zeta\left(2s-2\right)-1 \right).
\]
With
\[
\zeta_{H}\left(2s-2,1\right)=\zeta\left(2s-2\right),
\]
the final result is thus
\begin{align*}
\tilde{\zeta}_{\zc}\left(s\right) & =\frac{7}{4}\zeta\left(2s\right)-\zeta\left(2s-2\right)+\left(1-\frac{\pi^{2}}{6}\right)\zeta\left(2s-2\right)+\zeta\left(2s-2,2\right)\\
 & =\frac{7}{4}\zeta\left(2s\right)-\frac{\pi^{2}}{6}\zeta\left(2s-2\right)+\zeta\left(2s-2,2\right).
\end{align*}
\end{proof}

\begin{proof}[Proof of Theorem 4.4]
We compute 
\[
\frac{1}{\tilde{z}_{k}}=\sum_{i=1}^{k-1}\frac{1}{i\left(i+1\right)-k\left(k+1\right)}+\sum_{i=k+1}^{\infty}\left(\frac{1}{i\left(i+1\right)-k\left(k+1\right)}-\frac{1}{i\left(i+1\right)}\right).
\]
The first sum is
\begin{align*}
\sum_{i=1}^{k-1}\frac{1}{\left(i+\frac{1}{2}\right)^{2}-\left(k+\frac{1}{2}\right)^{2}} & =\sum_{i=1}^{k-1}\frac{1}{\left(i+k+1\right)\left(i-k\right)}\\&=\frac{1}{2k+1}\sum_{i=1}^{k-1}\left(\frac{1}{i-k}-\frac{1}{i+k+1}\right)\\
 & =\frac{1}{2k+1}\left(-\psi\left(k\right)+\psi\left(1\right)-\psi\left(2k+1\right)+\psi\left(k+2\right)\right).
\end{align*}
The second sum telescopes as
\begin{align*}
\sum_{i=k+1}^{\infty}\frac{1}{i\left(i+1\right)-k\left(k+1\right)}-&\left(\frac{1}{i}-\frac{1}{i+1}\right)  \\
&=\frac{1}{2k+1}\sum_{i=k+1}^{\infty}\left(\frac{1}{i-k}-\frac{1}{i+k+1}\right)-\frac{1}{k+1}\\
& =\frac{1}{2k+1}\left(\psi\left(2k+2\right)-\psi\left(1\right)\right)-\frac{1}{k+1},
\end{align*}
so that
\begin{align*}
\frac{1}{\tilde{z}_{k}} & =\frac{1}{2k+1}\left(-\psi\left(k\right)+\psi\left(2k+2\right)-\psi\left(2k+1\right)+\psi\left(k+2\right)\right)-\frac{1}{k+1}\\
 & =\frac{1}{2k+1}\left(\frac{1}{2k+1}+\frac{1}{k+1}+\frac{1}{k}\right)-\frac{1}{k+1}=\frac{1}{k}-\frac{2}{k+1}+\frac{1}{\left(2k+1\right)^{2}}.
\end{align*}
The complementary zeta function is computed as
\begin{align*}
\tilde{\zeta}_{\zc}\left(s\right) & =\sum_{n=1}^\infty\frac{1}{n^{s-1}\left(n+1\right)^{s-1}}\left(\frac{1}{n}-\frac{2}{n+1}+\frac{1}{\left(2n+1\right)^{2}}\right)\\
 & =\sum_{n=1}^\infty\frac{1}{n^{s}\left(n+1\right)^{s-1}}-2\sum_{n=1}^\infty\frac{1}{n^{s-1}\left(n+1\right)^{s}}+\sum_{n=1}^\infty\frac{1}{n^{s-1}\left(n+1\right)^{s-1}\left(2n+1\right)^{2}}.
\end{align*}
The first sum is computed using the partial fraction decomposition
\begin{align*}
\frac{1}{n^{s}\left(n+1\right)^{s-1}}&=\left(-1\right)^{s}\sum_{k=1}^{s}\left(-1\right)^{k}\binom{2s-2-k}{s-2}\frac{1}{n^{k}}+\left(-1\right)^{s}\sum_{k=1}^{s-1}\binom{2s-2-k}{s-1}\frac{1}{\left(n+1\right)^{k}} \\
&=\left(-1\right)^{s}\sum_{k=2}^{s}\left(-1\right)^{k}\binom{2s-2-k}{s-2}\frac{1}{n^{k}}+\left(-1\right)^{s}\sum_{k=2}^{s-1}\binom{2s-2-k}{s-1}\frac{1}{\left(n+1\right)^{k}} \\
&\qquad+ (-1)^s \binom{2s-3}{s-2} \left(  \frac{1}{n+1}-\frac{1}{n}  \right).
\end{align*}
We separated out the $k=1$ terms, and then used the symmetry $\binom{2s-3}{s-2}=\binom{2s-3}{s-1}$. Then
\begin{align*}
\sum_{n=1}^\infty\frac{1}{n^{s}\left(n+1\right)^{s-1}} & =\left(-1\right)^{s}
\left[
\sum_{k=2}^{s}\left(-1\right)^{k}\binom{2s-2-k}{s-2}\zeta\left(k\right)+\sum_{k=2}^{s-1}\binom{2s-2-k}{s-1}\left(\zeta\left(k\right)-1\right)  
\right] \\
&\qquad+ (-1)^s \binom{2s-3}{s-2}   \sum_{n=1}^\infty \left(  \frac{1}{n+1}-\frac{1}{n}  \right)\\
 & =\left(-1\right)^{s}\sum_{k=2}^{s}\left(-1\right)^{k}\binom{2s-2-k}{s-2}\zeta\left(k\right)+\left(-1\right)^{s}\sum_{k=2}^{s-1}\binom{2s-2-k}{s-1}\zeta\left(k\right)\\
 &\qquad -\left(-1\right)^{s}\sum_{k=1}^{s-1}\binom{2s-2-k}{s-1}\\
 & =\left(-1\right)^{s}\sum_{k=2}^{s}\left(\left(-1\right)^{k}\binom{2s-2-k}{s-2}+\binom{2s-2-k}{s-1}\right)\zeta\left(k\right) \\
&\qquad+\left(-1\right)^{s+1}\binom{2s-2}{s}.
\end{align*}
The second sum is also computed using the partial fraction decomposition
\begin{align*}
\frac{1}{n^{s-1}\left(n+1\right)^{s}}&=\left(-1\right)^{s+1}\sum_{k=1}^{s}\binom{2s-2-k}{s-2}\frac{1}{\left(n+1\right)^{k}}+\left(-1\right)^{s+1}\sum_{k=1}^{s-1}\left(-1\right)^{k}\binom{2s-2-k}{s-1}\frac{1}{n^{k}} \\
&=\left(-1\right)^{s+1}\sum_{k=2}^{s}\binom{2s-2-k}{s-2}\frac{1}{\left(n+1\right)^{k}}+\left(-1\right)^{s+1}\sum_{k=2}^{s-1}\left(-1\right)^{k}\binom{2s-2-k}{s-1}\frac{1}{n^{k}} \\
&\qquad+ (-1)^{s+1} \binom{2s-3}{s-2} \left(  \frac{1}{n+1}-\frac{1}{n}  \right).
\end{align*}
We deduce
\begin{align*}
\sum_{n=1}^\infty\frac{1}{n^{s-1}\left(n+1\right)^{s}} & =\left(-1\right)^{s+1}\sum_{k=2}^{s}\binom{2s-2-k}{s-2}\left(\zeta\left(k\right)-1\right)\\
&\qquad+\left(-1\right)^{s+1}\sum_{k=2}^{s-1}\left(-1\right)^{k}\binom{2s-2-k}{s-1}\zeta\left(k\right)\\
&\qquad+  (-1)^{s+1} \binom{2s-3}{s-2} \sum_{n=1}^\infty\left(  \frac{1}{n+1}-\frac{1}{n}  \right) \\
 & =\left(-1\right)^{s+1}\sum_{k=2}^{s}\left(\binom{2s-2-k}{s-2}+\left(-1\right)^{k}\binom{2s-2-k}{s-1}\right)\zeta\left(k\right)\\
&\qquad+\left(-1\right)^{s}\binom{2s-3}{s-2}.
\end{align*}
To compute the third sum, we use the partial fraction decomposition
\begin{align*}
\frac{1}{n^{s}\left(n+1\right)^{s}\left(2n+1\right)^{2}}&=\sum_{k=0}^{s}\frac{\alpha_{k}^{\left(s\right)}}{n^{k+1}}+\left(-1\right)^{s}\sum_{k=0}^{s}\frac{\beta_{k}^{\left(s\right)}}{\left(n+1\right)^{k+1}}+\frac{\left(-1\right)^{s}2^{2s}}{\left(2n+1\right)^{2}} \\
&=\sum_{k=1}^{s}\frac{\alpha_{k}^{\left(s\right)}}{n^{k+1}}+\left(-1\right)^{s}\sum_{k=1}^{s}\frac{\beta_{k}^{\left(s\right)}}{\left(n+1\right)^{k+1}}+\frac{\left(-1\right)^{s}2^{2s}}{\left(2n+1\right)^{2}} \\
&\qquad+ (-1)^{s-1}\beta_0^{(s)} \left( \frac{1}{n}-\frac{1}{n+1}\right),
\end{align*}
with residues
\[
\beta_{k}^{\left(s\right)}=\sum_{i=0}^{s-k-1}4^{i}\binom{2s-2i-k-2}{s-i-1},
\thinspace\thinspace
\alpha_{k}^{\left(s\right)}=\left(-1\right)^{s-k-1}\beta_{k}^{\left(s\right)}.
\]
We deduce 
\begin{align*}
\sum_{n=1}^\infty&\frac{1}{n^{s}\left(n+1\right)^{s}\left(2n+1\right)^{2}} \\
&=\sum_{k=1}^{s}\sum_{n=1}^\infty\frac{\alpha_{k}^{\left(s\right)}}{n^{k+1}}+\sum_{k=1}^{s}\sum_{n=1}^\infty\frac{\left(-1\right)^{s}\beta_{k}^{\left(s\right)}}{\left(n+1\right)^{k+1}}+\sum_{n=1}^\infty\frac{\left(-1\right)^{s}2^{2s}}{\left(2n+1\right)^{2}}\\
&\qquad+ (-1)^{s-1} \beta_0^{(s)}\sum_{n=1}^{\infty } \left( \frac{1}{n}-\frac{1}{n+1}\right) \\
 & =\left(-1\right)^{s}
\sum_{k=1}^{s}\left(\left(-1\right)^{k+1}\beta_{k}^{\left(s\right)}\zeta\left(k+1\right)+\beta_{k}^{\left(s\right)}\left(\zeta\left(k+1\right)-1\right)\right) \\
&\qquad+\left(-1\right)^{s}2^{2s}\sum_{n=1}^\infty\frac{1}{\left(2n+1\right)^{2}}-(-1)^{s}\beta_0^{(s)}\\
 & =\left(-1\right)^{s}\sum_{k=1}^{s}\left(1-\left(-1\right)^{k}\right)\beta_{k}^{\left(s\right)}\zeta\left(k+1\right)-\left(-1\right)^{s}\sum_{k=0}^{s}\beta_{k}^{\left(s\right)}
+\left(-1\right)^{s}2^{2s}\left(\frac{\pi^{2}}{8}-1\right).
\end{align*}

Since moreover it can be checked that
\[
\sum_{k=0}^{s}\beta_{k}^{\left(s\right)}=\frac{1}{2}\left(\frac{\left(2s+1\right)!}{s!s!}-2^{2s}\right),
\]
we deduce
\begin{align*}
\sum_{n=1}^\infty\frac{1}{n^{s}\left(n+1\right)^{s}\left(2n+1\right)^{2}} & =(-1)^{s}\sum_{k=1}^{s}\left(1-(-1)^{k}\right)\zeta(k+1)\beta_{k}^{\left(s\right)}\\
 & +(-1)^{s}\left(-\frac{(2s+1)!}{2(s!s!)}+\left(\frac{\pi^{2}}{8}-1\right)2^{2s}+2^{2s-1}\right).
\end{align*}
Putting the three terms together, we obtain
\begin{align*}
&\tilde{\zeta}_{\zc}\left(s\right)  =\sum_{n=1}^\infty\frac{1}{n^{s}\left(n+1\right)^{s-1}}-2\sum_{n=1}^\infty\frac{1}{n^{s-1}\left(n+1\right)^{s}}+\sum_{n=1}^\infty\frac{1}{n^{s-1}\left(n+1\right)^{s-1}\left(2n+1\right)^{2}}\\
 & =\left(-1\right)^{s}\sum_{k=2}^{s}\left(\left(-1\right)^{k}\binom{2s-2-k}{s-2}+\binom{2s-2-k}{s-1}\right)\zeta\left(k\right)+\left(-1\right)^{s+1}\binom{2s-2}{s}\\
&-2\left[\left(-1\right)^{s+1}\sum_{k=2}^{s}\left(\binom{2s-2-k}{s-2}+\left(-1\right)^{k}\binom{2s-2-k}{s-1}\right)\zeta\left(k\right)+\left(-1\right)^{s}2\binom{2s-3}{s-2}\right]\\
 & \qquad+(-1)^{s-1}\sum_{k=2}^{s}\left(1-(-1)^{k-1}\right)\zeta\left(k\right)\beta_{k-1}^{\left(s-1\right)} \\
&\qquad+(-1)^{s-1}\left(-\frac{\left(2s-1\right)!}{2\left(s-1\right)!\left(s-1\right)!}+\left(\frac{\pi^{2}}{8}-1\right)2^{2s-2}+2^{2s-3}\right),
\end{align*}
which can be simplified to
\begin{align*}
&\left(-1\right)^{s}\sum_{k=2}^{s}\zeta\left(k\right)\left(\left(-1\right)^{k}+2\right)\left( \binom{2s-2-k}{s-2}+\left(-1\right)^{k}\binom{2s-2-k}{s-1}\right)\\
&-\left(-1\right)^{s}\sum_{k=2}^{s}\zeta\left(k\right)\left(\left(1+\left(-1\right)^{k}\right)\beta_{k-1}^{\left(s-1\right)}\right)\\
&+\left(-1\right)^{s}\left[\frac{\left(2s-1\right)!}{2\left(s-1\right)!\left(s-1\right)!}-\binom{2s-2}{s}-4\binom{2s-3}{s-2}-\left(\frac{\pi^{2}}{8}-1\right)2^{2s-2}-2^{2s-3}\right],
\end{align*}
which completes the proof.
\end{proof}
\section{The Bessel zeta function}
\label{section:Bessel}
\subsection{The Bessel function case}

In this section, we study a non-elementary case of an extended zeta function based on the zeros of a transcendental function,
the Bessel function of the first kind  $J_{\nu}$ with parameter $\nu$, defined as
\[
J_{\nu}\left(z\right)=\frac{z^{\nu}}{2^{\nu}}\sum_{k=0}^\infty\frac{\left(-1\right)^{k}}{\Gamma\left(\nu+k+1\right)}\frac{\left(\frac{z^{2}}{4}\right)^{k}}{k!}.
\]
Its normalized version
\[
j_{\nu}\left(x\right)=2^{\nu}\Gamma\left(\nu+1\right)\frac{J_{\nu}\left(x\right)}{x^{\nu}}
\]
has Weierstrass factorization \cite[(8.544, Page 942)]{Table}
\[
j_{\nu}\left(z\right)=\prod_{k=1}^\infty\left(1-\frac{z^{2}}{x_{\nu,k}^{2}}\right),
\]
where $x_{\nu,k}$ are the real zeros of $j_{\nu}$, ordered by increasing
absolute values. We build from these zeros the Bessel zeta function
by choosing $z_{k}=x_{\nu,k}^{2},$ so that
\begin{equation}
\label{BesselZeta}
\zeta_{B,\nu}\left(2s\right):=\sum_{k=1}^\infty\frac{1}{z_{k}^{s}}=\sum_{k=1}^\infty\frac{1}{x_{\nu,k}^{2s}}.
\end{equation}
Note that this parameterized set of zeta functions 
includes several important special cases: when $\nu=\frac12$, we have $x_{\nu,k}=k\pi$, and we recover Riemann MZVs. When $\nu=-\frac12$, we have $x_{\nu,k}=\left(k+\frac12\right)\pi$, and we recover multiple $t$-values. A more in depth study of the Bessel zeta function, and the MZV type identities it satisfies, is given in \cite{Wakhare}. As shown in the following Lemma, it turns out that in the Bessel case, the complementary zeta function
can be explicitly computed as a linear combination of the Bessel zeta
function and of a depth $2$ Bessel MZV.
\begin{lemma}
The Bessel complementary zeta function is equal to
\begin{equation}
\tilde{\zeta}_{B,\nu}\left(s\right)=\frac{\nu+1}{2}\zeta_{B,\nu}\left(2s\right)-\zeta_{B,\nu}\left(2,2s-2\right).\label{eq:zetatildeBessel}
\end{equation}
\end{lemma}

\begin{proof}
The complementary zeros $\tilde{z}_{k}$ defined by \eqref{eq:1/ztilde 1} are computed
as
\begin{align*}
\frac{1}{\tilde{z}_{k}} & =\sum_{i=1}^{k-1}\frac{1}{x_{\nu,i}^{2}-x_{\nu,k}^{2}}+\sum_{i=k+1}^{\infty}\frac{1}{x_{\nu,i}^{2}-x_{\nu,k}^{2}}-\frac{1}{x_{\nu,i}^{2}}\\
 & =\sum_{\substack{i=1 \\ i \ne k}}^{\infty}\frac{1}{x_{\nu,i}^{2}-x_{\nu,k}^{2}}-\sum_{i\ge k+1}\frac{1}{x_{\nu,i}^{2}}
\end{align*}
since both sums converge separately.
The first sum can be expressed using the formula  by Calogero \cite{Calogero}:
\[
\sum_{\substack{i=1 \\ i \ne k}}^{\infty}\frac{1}{x_{\nu,i}^{2}-x_{\nu,k}^{2}}=\frac{\nu+1}{2}\frac{1}{x_{\nu,k}^{2}}.
\]
The complementary zeta function is then obtained as
\begin{align*}
\tilde{\zeta}_{B,\nu}\left(s\right) & =\sum_{k=1}^\infty\frac{1}{\tilde{z}_{k}z_{k}^{s-1}}=\sum_{k=1}^\infty\frac{1}{z_{k}^{s-1}}\left(\frac{\nu+1}{2}\frac{1}{z_{k}}-\sum_{i\ge k+1}\frac{1}{z_{i}}\right)\\
 & =\frac{\nu+1}{2}\zeta_{B,\nu}\left(2s\right)-\zeta_{B,\nu}\left(2,2s-2\right).
\end{align*}
\end{proof}


This result allows us to express some non-elementary identities about
Bessel MZVs as follows.
\begin{theorem}\label{thm510}
The Bessel MZV \eqref{BesselZeta} satisfies the identity
\[
\zeta_{B,\nu}\left(2s+6\right)=\frac{2}{\nu+1}\sum_{\substack{a+b=s \\ a\ge-1,\,b \ge 0}}\zeta_{B,\nu}\left(4+2a,2+2b\right),\,\,s \ge 0.
\]
\end{theorem}

\begin{proof}
Apply the
generalization of Euler's identity \eqref{zetak+3} and (\ref{eq:zetatildeBessel})
and some elementary algebra to obtain the identity. 
\end{proof}

\begin{corollary}\label{cor53}
We have
\[
\zeta_{B,\nu}\left(6\right) = \frac{2}{\nu +3}
\zeta_{B,\nu}\left(2\right)\zeta_{B,\nu}\left(4\right).
\]
\end{corollary}
\begin{proof}
Choosing $s=0$ in Theorem \ref{thm510} yields
\[
\zeta_{B,\nu}\left(6\right)=\frac{2}{\nu+1}\left(\zeta_{B,\nu}\left(2,4\right)+\zeta_{B,\nu}\left(4,2\right)\right).
\]
Since moreover
\[
\zeta_{B,\nu}\left(2,4\right)+\zeta_{B,\nu}\left(4,2\right) = \zeta_{B,\nu}\left(2\right)\zeta_{B,\nu}\left(4\right) - \zeta_{B,\nu}\left(6\right),
\]
the result follows after simple algebra.
\end{proof}
Notice that Corollary \ref{cor53} can be checked directly since the zeta values involved have respective explicit expressions
\[
\zeta_{B,\nu}\left( 2 \right)=\frac{1}{4\left( \nu+1 \right)},\,\,\zeta_{B,\nu}\left( 4 \right)=\frac{1}{16\left( 1+\nu \right)^2\left( 2+\nu \right)}
\]
and
\[
\zeta_{B,\nu}\left( 6 \right)=\frac{1}{32\left( 1+\nu \right)^3\left( 2+\nu \right)\left( 3+\nu \right)}.
\]
\subsection{The Bessel polynomial case}
\label{BesselPolynomial}
We conclude this series of examples with a case of a zeta function built on a {\it finite} sequence of numbers $\left\{ z_{k}\right\} $ chosen as 
the sequence of the zeros
$\left\{ z_{\nu,j}\right\}_{1 \le j \le n} $ of the Bessel polynomial $\theta_{n}\left(z\right)$ of degree $n=\nu-\frac{1}{2}:$
\[
\zeta_{\theta,n}  \left( s \right) = \sum_{j=1}^{n}
\frac{1}{z_{\nu,j}^{s}},
\] 
where we omit the dependence on $\nu$ in the notation as this parameter is fixed in what follows.

The Bessel polynomial $\theta_n$ is obtained from the modified Bessel function of the second kind $K_{\nu}$ with $\nu=n+\frac{1}{2}$,
as
\[
\theta_{n}\left(z\right)=\sqrt{\frac{2}{\pi}}e^{z}z^{n+\frac{1}{2}}K_{n+\frac{1}{2}}\left(z\right),
\]
or equivalently as 
\[
\theta_{n}\left(z\right)=\sum_{m=0}^{n}\frac{\left(n+m\right)!}{2^{m}\left(n-m\right)!m!}z^{n-m}.
\]
For example,
\[
\theta_{0}\left( z \right)=1,\,\,\theta_{1}\left( z \right)=1+z, \,\,\theta_{2}\left( z \right)= 3 +z+z^2.
\]
These roots are complex conjugated and satisfy the identity \cite[2.10b]{Ahmed}
\[
\sum_{\substack{k=1\\
k\ne j
}
}^{n}\frac{1}{z_{\nu,k}-z_{\nu,j}}=-1-\frac{\nu-\frac{1}{2}}{z_{\nu,j}}.
\]
We state the following theorem and omit its proof since it follows the same steps as the previous one. 
\begin{theorem}
In the case where $\left\{ z_{i}\right\} $ are chosen as the zeros of the Bessel polynomial of degree $n$, and with $\nu =n + \frac{1}{2},$ the sequence of complementary zeros is
\[
\frac{1}{\tilde{z}_{\nu,k}}=-1-\frac{\nu-\frac{1}{2}}{z_{\nu,j}}-\sum_{j= k+1}^n\frac{1}{z_{\nu,j}},
\]
and the complementary MZV is
\[
\tilde{\zeta}_{\theta,n}\left(s\right)=\left(\frac{1}{2}-\nu\right)\zeta_{\theta,n}\left(s\right)-\zeta_{\theta,n}\left(s-1\right)-\zeta_{\theta,n}\left(1,s-1\right).
\]
As a consequence,
\[
\zeta_{\theta,n}\left(s+3\right)=\frac{2}{1-2\nu}\left[\zeta_{\theta,n}\left(s+2\right)+\sum_{\substack{a\ge-1\\
a+b=s
}
}\zeta_{\theta,n}\left(2+a,1+b\right)\right],
\]
with, as the special case $s=0,$
\[
\zeta_{\theta,n}\left(2\right)+\zeta_{\theta,n}\left(1\right)\zeta_{\theta,n}\left(2\right)=\left(\frac{3}{2}-\nu\right)\zeta_{\theta,n}\left(3\right).
\]
\end{theorem}

\section{Conclusion}
We have shown that several usual identities for MZVs are naturally of structural type, i.e. can be extended to MZVs built from an arbitrary sequence of nonzero numbers $\zc$. The price to pay is the use of the complementary sequence of numbers defined by \eqref{eq:1/ztilde 1} or \eqref{eq:1/ztilde 1-1} in the multivariate case. However, the natural appearance of the complementary zeta function through multiple methods suggests that it is somehow fundamental to the ring of quasisymmetric functions; it is an open question to characterize this in terms of known bases for the ring of quasisymmetric functions.

Additionally, exploring further special cases of $\tilde{z}_k$ appears to lead to many new sum relations. The case $z_k=k^{2n}$ for positive integral $n$ appears particularly promising, as it leads to nonlinear dependence relations among Riemann MZVs. Many of our results are also crying out for $q$-analogs; we believe that with the correct choice of $z_k$, we will have $\tilde{\zeta}_{\zc}$ reduce to some known $q$-multiple zeta function. Finally, we are still unsure of what the ``correct" higher order analog of $\tilde{\zeta}_\zc(s)$ is, since it appears to only be the structural analog of a depth $2$ MZV.

\section*{Acknowledgments}

The authors thank Karl Dilcher for his generous invitation to attend the Eighteenth International Conference on Fibonacci Numbers and Their Applications in Halifax, Canada, in July 2018.
This paper is dedicated to the memory of J.M. Borwein, the great mathemagician.

\end{document}